\documentclass[a4paper]{article}

\usepackage{amsmath,amsthm,amssymb}
\usepackage{graphicx,subcaption, tikz}
\usepackage{comment,xspace,paralist}

\usepackage[shortlabels]{enumitem}	
\usepackage{hyperref}       
\hypersetup{
    colorlinks=true,
    citecolor = blue,
    linkcolor=black,
    filecolor=magenta,
    urlcolor=cyan,
    bookmarks=true,
}
\usepackage{times}
\usepackage{todonotes}
\usepackage[ruled,linesnumbered]{algorithm2e}
\usepackage{enumerate}
\newtheorem{theorem}{Theorem}
\newtheorem{lemma}{Lemma}
\newtheorem{claim}{Claim}

\title{\bf  Vertex-Critical $(P_5, chair)$-Free Graphs}

\author{
Shenwei Huang\thanks{College of Computer Science, Nankai University, Tianjin 300350, China. Email: \texttt{shenweihuang@nankai.edu.cn.} Supported by Natural Science Foundation of Tianjin (20JCYBJC01190).} \thanks{Tianjin Key Laboratory of Network and Data Security Technology, Nankai University, Tianjin 300071, China.}
\and
Zeyu Li\thanks{College of Computer Science, Nankai University, Tianjin 300350, China.} \thanks{Tianjin Key Laboratory of Network and Data Security Technology, Nankai University, Tianjin 300071, China.}
}

\date{January 4, 2022}

\begin{document}

\maketitle

\begin{abstract} 
Given two graphs $H_1$ and $H_2$, a graph $G$ is $(H_1,H_2)$-free if it contains no induced subgraph isomorphic to $H_1$ or $H_2$. 
A $P_t$ is the path on $t$ vertices. A chair is a $P_4$ with an additional vertex adjacent to one of the middle vertices of the $P_4$.
A graph $G$ is $k$-vertex-critical if $G$ has chromatic number $k$ but every proper induced subgraph of $G$ has chromatic number less than $k$.  In this paper, we prove that there are finitely many 5-vertex-critical $(P_5,chair)$-free graphs. 
\end{abstract}


\section{Introduction}

All graphs in this paper are finite and simple.
 We say that a graph $G$ {\em contains} a graph $H$ if $H$ is
isomorphic to an induced subgraph of $G$.  A graph $G$ is
{\em $H$-free} if it does not contain $H$. 
For a family of graphs $\mathcal{H}$,
$G$ is {\em $\mathcal{H}$-free} if $G$ is $H$-free for every $H\in \mathcal{H}$.
When $\mathcal{H}$ consists of two graphs, we write
$(H_1,H_2)$-free instead of $\{H_1,H_2\}$-free.
As usual, $P_t$ and $C_s$ denote
the path on $t$ vertices and the cycle on $s$ vertices, respectively.
A {\em clique} (resp.\ {\em independent set}) in a graph is a set of pairwise adjacent (resp.\ nonadjacent) vertices. The complete
graph on $n$ vertices is denoted by $K_n$.
The graph $K_3$ is also referred to as the {\em triangle}.
The \emph{clique number} of $G$, denoted by $\omega(G)$, is the size of a largest clique in $G$.
For two graphs $G$ and $H$, we use $G+H$ to denote the \emph{disjoint union} of $G$ and $H$.
If a graph $G$ can be partitioned into $k$ independent sets $S_1,\ldots,S_k$ such that there is an edge between every vertex in $S_i$
and every vertex in $S_j$ for all $1\le i<j\le k$, $G$ is called a {\em complete $k$-partite graph}; each $S_i$ is called a {\em part}
of $G$. If we do not specify the number of parts in $G$, we simply say that $G$ is a {\em complete multipartite graph}.
We denote by $K_{n_1,\ldots,n_k}$ the complete $k$-partite graph such that the $i$th part $S_i$ has size $n_i$, for each $1\le i\le k$.

A \emph{$q$-coloring} of a graph $G$ is a function $\phi:V(G)\longrightarrow \{ 1, \ldots ,q\}$ such that
$\phi(u)\neq \phi(v)$ whenever $u$ and $v$ are adjacent in $G$.
And a $q$-coloring of $G$ is also a partition of $V(G)$ into $q$ independent sets.
A graph is {\em $q$-colorable} if it admits a $q$-coloring.
The \emph{chromatic number} of a graph $G$, denoted by
$\chi (G)$, is the minimum number $q$ for which $G$ is $q$-colorable.
We call a graph $G$ is {\em $k$-chromatic} when $\chi(G)=k$. 


A graph $G$ is {\em $k$-critical} if it is
$k$-chromatic and $\chi(G-e)<\chi(G)$ for any edge $e\in E(G)$.
We call a graph is {\em critical} if it is $k$-critical for some integer $k\ge 1$.
A graph $G$ is {\em $k$-vertex-critical} if $\chi(G)=k$ and $\chi(G-v)<k$ for any $v\in V(G)$.
For a set $\mathcal{H}$ of graphs and a graph $G$, we say that $G$ is {\em $k$-vertex-critical $\mathcal{H}$-free}
if it is $k$-vertex-critical and $\mathcal{H}$-free. Our research is mainly motivated by the following theorems.

\begin{theorem}[\cite{HMRSV15}]\label{thm:P5infinite}
For any fixed $k\ge 5$, there are infinitely many $k$-vertex-critical $P_5$-free graphs.
\end{theorem}

Thus, it is natural to consider which subclasses of $P_5$-free graphs have finitely many $k$-vertex-critical graphs.
The reason for finiteness is that if we know there are only finitely many $k$-vertex-critical graphs, then there is a polynomial-time algorithm for $(k-1)$-coloring graphs in that class. In 2021, Kameron, Goedgebeur, Huang and Shi \cite{CGHS21} obtained the following dichotomy
result for $k$-vertex-critical $(P_5, H)$-free graphs  when $|H|=4$.

\begin{theorem}[\cite{CGHS21}]
Let $H$ be a graph of order $4$ and $k\geq5$ be a fixed integer. Then there are infinitely many $k$-vertex-critical $(P_5, H)$-free graphs if and only if $H$ is $2P_2$ or $P_1+K_3$. 
\end{theorem}

In \cite{CGHS21}, it was also asked which five-vertex graphs $H$ can lead to finitely many $k$-vertex-critical $(P_5,H)$-free graphs. 
It is known that there are finitely many 5-vertex-critical ($P_5$,banner)-free graphs~\cite{CHLS19, HLS19}, and
finitely many $k$-vertex-critical $(P_5,\overline{P_5})$-free graphs for every fixed $k$~\cite{DHHMMP17}.
Hell and Huang proved that there are finitely many $k$-vertex-critical $(P_6,C_4)$-free graphs~\cite{HH17}.
This was later generalized to $(P_t,K_{r,s})$-free graphs in the context of $H$-coloring~\cite{KP19}.
This gives an affirmative answer for $H=K_{2,3}$. Recently, it was also shown that the answer to the above question is positive if $H$ is gem or $\overline{P_2+P_3}$ \cite{CQH21}. Moreover, it was proved that there are finitely many 5-vertex-critical $(P_5,bull)$-free
graphs \cite{HL22}. 

In this article, we continue such a study. 
A chair is a $P_4$ with an additional vertex adjacent to one of the middle vertices of the $P_4$ (see \autoref{fig:chair}).
In particular, we prove the following.
\begin{theorem}\label{thm:chair}
There are finitely many 5-vertex-critical $(P_5,chair)$-free graphs.
\end{theorem}

\begin{figure}
    \centering
\begin{tikzpicture}[line width = 1pt, vertex/.style={minimum size=2pt,fill=none,draw,circle},auto]
\node [vertex] (v1) at (0,-1){};
\node [vertex] (v2) at (0,0){} ;
\node [vertex] (v3) at (1,0){} ;
\node [vertex] (v4) at (1,-1){};
\node [vertex] (v5) at (1,1){};
\draw (v1)--(v2)--(v3)--(v4);
\draw (v3)--(v5);
\end{tikzpicture}
    \caption{The graph chair.}
    \label{fig:chair}
\end{figure}
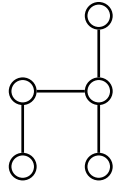

\section{Preliminaries}\label{sec:pre}

For general graph theory notation we follow~\cite{BM08}.
Let $G=(V,E)$ be a graph.  
If $uv\in E$, we say that $u$ and $v$ are {\em neighbors} or {\em adjacent}; 
otherwise $u$ and $v$ are {\em nonneighbors} or {\em nonadjacent}. We use $u\sim v$ to mean that $u$ and $v$ are neighbors
and $u\nsim v$ to mean that $u$ and $v$ are nonneighbors.
The \emph{neighborhood} of a vertex $v$, denoted by $N_G(v)$, is the set of neighbors of $v$.
For a set $X\subseteq V(G)$, let $N_G(X)=\bigcup_{v\in X}N_G(v)\setminus X$.
We shall omit the subscript whenever the context is clear.
For $X,Y\subseteq V$, we say that $X$ is \emph{complete} (resp.\ \emph{anticomplete}) to $Y$
if every vertex in $X$ is adjacent (resp.\ nonadjacent) to every vertex in $Y$.
If $X=\{x\}$, we write ``$x$ is complete (resp.\ anticomplete) to $Y$'' instead of ``$\{x\}$ is complete (resp.\ anticomplete) to $Y$''.
If a vertex $v$ is neither complete nor anticomplete to a set $S$, we say that $v$ is {\em mixed} on $S$.
If a vertex $v$ is neither complete nor anticomplete to two ends of an edge, we say that $v$ is {\em distinguish} the edge.
We say that $H$ is a {\em homogeneous} set if no vertex in $V-H$ is mixed on $H$.
More generally, we say that $H$ is {\em homogeneous} with respect to a subset $S\subseteq V$ if no vertex in $S$ can be mixed on $H$.
 For $S\subseteq V$, the subgraph \emph{induced} by $S$, is denoted by $G[S]$.


A pair of {\em comparable vertices} of $G$ is pairwise nonadjacent vertices $u,v$ such that $N(v)\subseteq{N(u)}$ or $N(u)\subseteq{N(v)}$. It is well-known that $k$-vertex-critical graphs cannot contain comparable vertices.
We shall use the following generalization in later proofs.
 
\begin{lemma}[\cite{CGHS21}]\label{lem:dominated subsets}
Let $G$ be a $k$-vertex-critical graph. Then $G$ has no two nonempty disjoint subsets $X$ and $Y$ of $V(G)$ that satisfy all the following conditions.

\begin{itemize}
\item $X$ and $Y$ are anticomplete to each other.
\item $\chi(G[X])\le \chi(G[Y])$.
\item $Y$ is complete to $N(X)$.
\end{itemize}
\end{lemma}

\section{New Results}\label{sec:new}

In this section, we prove our new results: there are finitely many 5-vertex-critical $(P_5,chair)$-free graphs. 
To prove \autoref{thm:chair}, we prove the following.

\begin{theorem}\label{thm:containing a C5}
Let $G$ be a 5-vertex-critical $(P_5,chair)$-free graph. If $G$ contains a $C_5$, then $G$ has finite order.
\end{theorem}

\begin{proof}[Proof of \autoref{thm:chair} assuming \autoref{thm:containing a C5}]
Let $G$ be a 5-vertex-critical $(P_5,chair)$-free graph. If $G$ contains $C_5$, then $G$ has finite order
by \autoref{thm:containing a C5}. If $G$ is $C_5$-free, then $G$ has finite order by a result in \cite{HMRSV15}
that there are only thirteen 5-vertex-critical $(P_5,C_5)$-free graphs. In either case, $G$ has finite order.
This completes the proof.
\end{proof}

Next we prove \autoref{thm:containing a C5}.

\subsection{Structure Around $C_5$}
In this subsection, we discuss some structural properties of $(P_5,chair)$-free graphs containing a $C_5$. 
Let $G$ be a connected $(P_5,chair)$-free graph containing an induced $C_5$.
Let $C=v_1, v_2, v_3, v_4, v_5$ be an induced $C_5$ with $v_iv_{i+1}$ being an edge. 
We divide $V\backslash{V(C)}$ as follows, where all indices are modulo 5.

\begin{itemize}
    \item []$S_0=\{v\in{V\backslash{V(C)}}:N_C(v)=\varnothing\}$,
    \item []$S_1(i)=\{v\in{V\backslash{V(C)}}:N_C(v)=\{v_i\}\}$,
    \item []$S^1_2(i)=\{v\in{V\backslash{V(C)}}:N_C(v)=\{v_i,v_{i+1}\}\}$,
    \item []$S^2_2(i)=\{v\in{V\backslash{V(C)}}:N_C(v)=\{v_i,v_{i+2}\}\}$,
    \item []$S^1_3(i)=\{v\in{V\backslash{V(C)}}:N_C(v)=\{v_{i-1},v_i,v_{i+1}\}\}$,
    \item []$S^2_3(i)=\{v\in{V\backslash{V(C)}}:N_C(v)=\{v_{i-2},v_i,v_{i+2}\}\}$,
    \item []$S_4(i)=\{v\in{V\backslash{V(C)}}:N_C(v)=\{v_{i-2},v_{i-1},v_{i+1},v_{i+2}\}\}$,
    \item []$S_5=\{v\in{V\backslash{V(C)}}:N_C(v)=V(C)\}$.
\end{itemize}

We use $S^m_3(i\pm{1})$ to denote $S^m_3(i+1)\cup{S^m_3(i-1)}$ for $m=1,2$.
The notations $S^m_3(i\pm{2})$, $S_4(i\pm{1})$ and $S_4(i\pm{2})$ are defined similarly.
We now prove some properties about these sets.

\begin{claim}\label{clm:S12}
    $S_1(i)\cup{S^1_2(i)}\cup{S^2_2(i)}=\varnothing$, for all $1\leq i\leq5$.
\end{claim}

\begin{proof}
    Suppose not. 
    Let $u,v$ be arbitrary two vertices such that $v\in{S_1(i)\cup{S^1_2(i)}}$, $u\in{S^2_2(i)}$.
    Then $\{v,v_i,v_{i-1},v_{i-2},v_{i-3}\}$ induces a $P_5$, and $\{u,v_i,v_{i-1},v_{i-2}\}$ and $\{v_{i+1}\}$ induce a chair.
\end{proof}

\begin{claim}\label{clm:S0}
    $S_0=\varnothing$.
\end{claim}

\begin{proof}
   Suppose not. 
   We will first show that $N(S_0)\subseteq S_5$. 
   Since $G$ is connected, there is a pair of vertices $u$ and $v$ such that $u\in{S_0}, v\in{V(G)\backslash{S_0}}$ and $u\sim{v}$.
   If $v\in{S^1_3(i)}$ for any $i$, then $\{u,v,v_{i+1},v_{i+2},v_{i-2}\}$ induces a $P_5$, a contradiction. 
   If $v\in{S^2_3(i)\cup{S_4(i+1)}}$ for any $i$, then $\{v_{i+1},v_i,v,v_{i-2}\}$ and $\{u\}$ induce a chair, a contradiction. 
   Thus, $v$ can only belong to $S_5$. 
   Then, two nonempty disjoint subsets $S_0$ and $C$ of $V(G)$ satisfy the three conditions of \autoref{lem:dominated subsets}, 
   a contradiction. Therefore, $S_0=\varnothing$.
\end{proof}

\begin{claim}\label{clm:S13}
    $S^1_3(i)$ is clique, for all $1\leq i\leq5$.
\end{claim}

\begin{proof}
    Suppose not. We assume that there are two vertices $u,v\in{S^1_3(i)}$ with $u\nsim{v}$. Then $\{v, v_{i+1},v_{i+2},v_{i-2}\}$ and $\{u\}$ induce a chair in $G$, a contradiction. 
\end{proof}

\begin{claim}\label{clm:45-3}
    Each vertex in $S_4(i)\cup{S_5}$ is either complete or anticomplete to a component of $S^2_3(i)$, for all $1\leq i\leq5$.
\end{claim}

\begin{proof}
    We assume that there is an edge $uv$ of $S^2_3(i)$ can be distinguished by vertex $s\in{S_4(i)\cup{S_5}}$. Without loss of generality, let $s\sim{u}$, $s\nsim{v}$. Then $\{v_{i-1},s,u,v\}$ and $\{v_{i+1}\}$ induce a chair.
\end{proof}

\begin{claim}\label{clm:other-3}
     Each vertex in $V(G)-(S^2_3(i)\cup{S_4(i)}\cup{S_5})$ is either complete or anticomplete to $S^2_3(i)$, for all $1\leq i\leq5$.
\end{claim}

\begin{proof}
By symmetry, it suffices to prove the claim for $i,i+1$ and $i+2$. Let $v\in{S^2_3(i)}$. 
If $v$ is adjacent to $s_1\in{S^1_3(i+1)}$, then $\{v_{i-1},v_{i-2},v,s_1,v_{i+1}\}$ is an induced $P_5$. 
If $v$ is not adjacent to $s_2\in{S^1_3(i)\cup{S^2_3(i+1)}\cup{S_4(i+2)}}$, then $\{v_{i-1},s_2,v_{i+1},v_{i+2},v\}$  is an induced $P_5$. 
If $v$ is not adjacent to $s_3\in{S^2_3(i+2)\cup{S_4(i+1)}}$, then $\{v_{i-1},s_3,v_{i+2},v\}$ and $\{v_{i+1}\}$ induce a chair. 
If $v$ is not adjacent to $s_4\in{S^1_3(i+2)}$, then $\{v_{i-1},v_i,v_{i+1},v\}$ and $\{s_4\}$ induce a chair. 
\end{proof}

\begin{claim}\label{clm:homo}
    Every component of $S^2_3(i)$ is a homogeneous set.
\end{claim}

\begin{proof}
  By \autoref{clm:45-3} and \autoref{clm:other-3}, there is no vertex of $G\backslash S^2_3(i)$ that can distinguish an edge of $S^2_3(i)$. 
\end{proof}

Let $T_i=S^1_3(i\pm{2})\cup S^2_3(i\pm{1})\cup S^2_3(i\pm{2})$ for each $i$. 

\begin{claim}\label{clm:S4-1}
    $S_4(i)$ is complete to $T_i$, for all $1\leq i\leq5$.
\end{claim}

\begin{proof}
By the symmetry, it suffers to prove the claim for $S^1_3(i+2)\cup S^2_3(i+1)\cup S^2_3(i+2)$.
Let $v\in{S_4(i)}$.
If $v$ is not adjacent to $s_1\in{S^1_3(i+2})$, then $\{v_i,v_{i-1},v,v_{i+2},s_1\}$ induces a $P_5$, a contradiction.
If $v$ is not adjacent to $s_2\in{S^2_3(i+1)}$, then $\{v_i,v_{i-1},v,v_{i+2}\}$ and $\{s_2\}$ induce a chair, a contradiction. 
If $v$ is not adjacent to $s_3\in{S^2_3(i+2})$, then $\{s_3,v_i,v_{i+1},v,v_{i-2}\}$ induces a $P_5$, a contradiction.
\end{proof}

\begin{claim}\label{clm:S4-2}
    For each $s\in{S^1_3(i)\cup{S_4(i\pm{2})}}$, $u,v\in S_4(i)$ with $uv\notin{E}$, $s$ cannot mix on $\{u,v\}$, for all $1\leq i\leq5$.
\end{claim}

\begin{proof}
    By the symmetry, it suffers to prove the claim for $S^1_3(i)\cup S_4(i+2)$.
    Let $s\in S^1_3(i)\cup S_4(i+2)$ with $s\sim{u}$, $s\nsim{v}$ , then $\{v_i,s,u,v_{i+2},v\}$ induces a $P_5$. 
\end{proof}
      
     Let $R_i=S^1_3(i\pm{1})\cup S^2_3(i)\cup S_4(i\pm{1})\cup S_5$, for each $i$. 
      
\begin{claim}\label{clm:S4-3}
    For each $s\in{R_i}$, $u,v\in S_4(i)$ with $uv\notin{E}$, $s$ is adjacent to at least one of $\{u,v\}$, for all $1\leq i\leq5$.
\end{claim} 

\begin{proof}
   By the symmetry, it suffers to prove the claim for $S^1_3(i+1)\cup S^2_3(i)\cup S_4(i+1)\cup S_5$.
   Let $s_1\in{S^1_3(i+1)\cup S^2_3(i)\cup S_4(i-1)}$, if $s_1$ is nonadjacent to both $\{u,v\}$, then $\{v,v_{i-1},v_i,s_1\}$ and $\{u\}$ induce a chair.
   Let $s_2\in{S_5}$, if $s_2$ is nonadjacent to both $\{u,v\}$, then $\{v_i,s_2,v_{i-2},v\}$ and $\{u\}$ induce a chair.   
\end{proof}

\begin{claim}\label{clm:S4(i+2)}
  Every vertex in $S_4(i\pm{2})$ is complete to $x,y\in{S_4(i)}$ with $xy\notin{E}$.
\end{claim}

\begin{proof}
     By symmetry, let $v\in{S_4(i+2)}$. 
      $v$ can not mix on $x,y$ by \autoref{clm:S4-2}.
      If $v\nsim x$ and $v\nsim y$, $\{v_i,v,v_{i-2},x\}$ and $\{y\}$ induce a chair. Then $v$ is complete to $\{x,y\}$.
\end{proof}

\subsection{Proof of \autoref{thm:containing a C5}}

Let graph family $\mathcal{F}=\{K_5, W, P, Q_1, Q_2, Q_3\}$ (see \autoref{fig:F1F2}).
The adjacency lists of $\mathcal{F}$ are given in the Appendix.  
It is routine to verify that every graph in $\mathcal{F}$ is a 5-vertex-critical $(P_5, chair)$-free graph.

\begin{proof}[Proof of \autoref{thm:containing a C5}]
Let $G$ be a 5-vertex-critical $(P_5,chair)$-free graph.
If $G$ contains a induced $F\in{\mathcal{F}}$, then $G$ is isomorphic to $F$ since $G$ is 5-vertex-critical. 
Therefore, we may assume that $G$ is $\mathcal{F}$-free.

 By \autoref{clm:S12} and \autoref{clm:S0}, $G$ has a finite order if and only if $S_3\cup{S_4}\cup{S_5}$ has finite size.

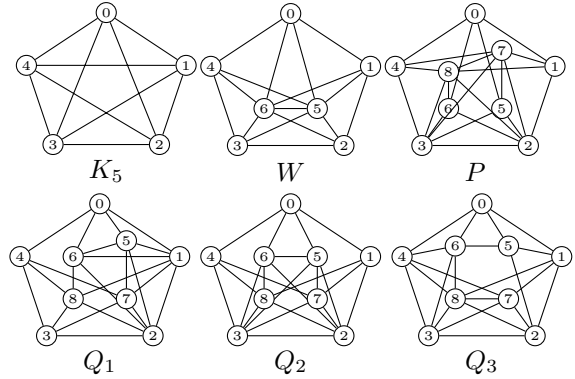
\begin{figure}[tb]
\centering
\begin{subfigure}{0.2\textwidth}
\centering
\begin{tikzpicture}[scale=0.7]
\tikzstyle{vertex}=[draw, circle, fill=white!100, minimum width=1pt,inner sep=1pt]

\node [vertex] (v1) at (0,1) {\tiny $0$};
\node [vertex] (v2) at (1.5,0) {\tiny $1$};
\node [vertex] (v3) at (1,-1.5) {\tiny $2$};
\node [vertex] (v4) at (-1,-1.5) {\tiny $3$};
\node [vertex] (v5) at (-1.5,0) {\tiny $4$};
\draw (v1)--(v2)--(v3)--(v4)--(v5)--(v1);
\draw (v1)--(v4)--(v2)--(v5)--(v3)--(v1);

\node at (0,-2) {$K_5$};
\end{tikzpicture}
\end{subfigure}%
\begin{subfigure}{0.2\textwidth}
\centering
\begin{tikzpicture}[scale=0.7]
\tikzstyle{vertex}=[draw, circle, fill=white!100, minimum width=1pt,inner sep=1pt]

\node [vertex] (v1) at (0,1) {\tiny $0$};
\node [vertex] (v2) at (1.5,0) {\tiny $1$};
\node [vertex] (v3) at (1,-1.5) {\tiny $2$};
\node [vertex] (v4) at (-1,-1.5) {\tiny $3$};
\node [vertex] (v5) at (-1.5,0) {\tiny $4$};
\draw (v1)--(v2)--(v3)--(v4)--(v5)--(v1);
\node [vertex] (x) at (0.5,-0.8) {\tiny $5$};
\draw (v1)--(x)  (v2)--(x) (v3)--(x) (v4)--(x) (v5)--(x);
\node [vertex] (y) at (-0.5,-0.8) {\tiny $6$};
\draw (v1)--(y)  (v2)--(y) (v3)--(y) (v4)--(y) (v5)--(y) (x)--(y);

\node at (0,-2) {$W$};
\end{tikzpicture}
\end{subfigure}%
\begin{subfigure}{0.2\textwidth}
\centering
\begin{tikzpicture}[scale=0.7]
\tikzstyle{vertex}=[draw, circle, fill=white!100, minimum width=1pt,inner sep=1pt]

\node [vertex] (v1) at (0,1) {\tiny $0$};
\node [vertex] (v2) at (1.5,0) {\tiny $1$};
\node [vertex] (v3) at (1,-1.5) {\tiny $2$};
\node [vertex] (v4) at (-1,-1.5) {\tiny $3$};
\node [vertex] (v5) at (-1.5,0) {\tiny $4$};
\draw (v1)--(v2)--(v3)--(v4)--(v5)--(v1);

\node[vertex] (y) at (0.5,-0.8) {\tiny $5$};
\draw (v1)--(y)  (v3)--(y) (v4)--(y);

\node[vertex] (z) at (-0.5,-0.8) {\tiny $6$};
\draw (v1)--(z) (v3)--(z)  (v4)--(z);

\node[vertex] (x) at (0.5,0.3) {\tiny $7$};
\draw (v2)--(x)  (v3)--(x) (v4)--(x) (v5)--(x);

\node[vertex] (t) at (-0.5,-0.1) {\tiny $8$};
\draw (v2)--(t) (v3)--(t)  (v4)--(t) (v5)--(t);

\draw (y)--(x)--(t)--(z);

\node at (0,-2) {$P$};
\end{tikzpicture}
\end{subfigure}

\begin{subfigure}{0.2\textwidth}
\centering
\begin{tikzpicture}[scale=0.7]
\tikzstyle{vertex}=[draw, circle, fill=white!100, minimum width=1pt,inner sep=1pt]

\node [vertex] (v1) at (0,1) {\tiny $0$};
\node [vertex] (v2) at (1.5,0) {\tiny $1$};
\node [vertex] (v3) at (1,-1.5) {\tiny $2$};
\node [vertex] (v4) at (-1,-1.5) {\tiny $3$};
\node [vertex] (v5) at (-1.5,0) {\tiny $4$};
\draw (v1)--(v2)--(v3)--(v4)--(v5)--(v1);

\node[vertex] (y) at (0.5,0.3) {\tiny $5$};
\draw (v1)--(y)  (v2)--(y) (v3)--(y);

\node[vertex] (z) at (-0.5,0) {\tiny $6$};
\draw (v1)--(z) (v2)--(z)  (v3)--(z);

\node[vertex] (x) at (0.5,-0.8) {\tiny $7$};
\draw (v2)--(x)  (v3)--(x) (v4)--(x) (v5)--(x);

\node[vertex] (t) at (-0.5,-0.8) {\tiny $8$};
\draw (v2)--(t) (v3)--(t)  (v4)--(t) (v5)--(t);

\draw (x)--(y)--(z)--(t);

\node at (0,-2) {$Q_1$};
\end{tikzpicture}
\end{subfigure}
\begin{subfigure}{0.2\textwidth}
\centering
\begin{tikzpicture}[scale=0.7]
\tikzstyle{vertex}=[draw, circle, fill=white!100, minimum width=1pt,inner sep=1pt]

\node [vertex] (v1) at (0,1) {\tiny $0$};
\node [vertex] (v2) at (1.5,0) {\tiny $1$};
\node [vertex] (v3) at (1,-1.5) {\tiny $2$};
\node [vertex] (v4) at (-1,-1.5) {\tiny $3$};
\node [vertex] (v5) at (-1.5,0) {\tiny $4$};
\draw (v1)--(v2)--(v3)--(v4)--(v5)--(v1);

\node[vertex] (y) at (0.5,0) {\tiny $5$};
\draw (v1)--(y)  (v3)--(y) (v4)--(y);

\node[vertex] (z) at (-0.5,0) {\tiny $6$};
\draw (v1)--(z) (v3)--(z)  (v4)--(z);

\node[vertex] (x) at (0.5,-0.8) {\tiny $7$};
\draw (v2)--(x)  (v3)--(x) (v4)--(x) (v5)--(x);

\node[vertex] (t) at (-0.5,-0.8) {\tiny $8$};
\draw (v2)--(t) (v3)--(t)  (v4)--(t) (v5)--(t);

\draw (x)--(y)--(z)--(t);

\node at (0,-2) {$Q_2$};
\end{tikzpicture}
\end{subfigure}
\begin{subfigure}{0.2\textwidth}
\centering
\begin{tikzpicture}[scale=0.7]
\tikzstyle{vertex}=[draw, circle, fill=white!100, minimum width=1pt,inner sep=1pt]

\node [vertex] (v1) at (0,1) {\tiny $0$};
\node [vertex] (v2) at (1.5,0) {\tiny $1$};
\node [vertex] (v3) at (1,-1.5) {\tiny $2$};
\node [vertex] (v4) at (-1,-1.5) {\tiny $3$};
\node [vertex] (v5) at (-1.5,0) {\tiny $4$};
\draw (v1)--(v2)--(v3)--(v4)--(v5)--(v1);

\node[vertex] (y) at (0.5,0.2) {\tiny $5$};
\draw (v1)--(y)  (v3)--(y) (v2)--(y);

\node[vertex] (z) at (-0.5,0.2) {\tiny $6$};
\draw (v1)--(z) (v5)--(z)  (v4)--(z);

\node[vertex] (x) at (0.5,-0.8) {\tiny $7$};
\draw (v2)--(x)  (v3)--(x) (v4)--(x) (v5)--(x);

\node[vertex] (t) at (-0.5,-0.8) {\tiny $8$};
\draw (v2)--(t) (v3)--(t)  (v4)--(t) (v5)--(t);

\draw (x)--(t)--(z)--(y);

\node at (0,-2) {$Q_3$};
\end{tikzpicture}
\end{subfigure}
\caption{Graph Family $\mathcal{F}$.}
\label{fig:F1F2}
\end{figure}

\begin{claim}\label{clm:S13<3}
    $|S^1_3(i)|\leq2$, for all $1\leq i\leq5$.
\end{claim}

\begin{proof}
	If $|S^1_3(i)|\ge 3$, then $S^1_3(i)\cup \{v_i,v_{i+1}\}$ contains a $K_5$ by \autoref{clm:S13}, a contradiction.
\end{proof}

\begin{claim}\label{clm:clique}
   $\chi(S^2_3(i)\cup{S_4(i)}\cup{S_5})\leq 2$, for all $1\leq i\leq5$. 
\end{claim}

\begin{proof}
    If $\chi({S^2_3(i)\cup{S_4(i)}\cup{S_5}})\geq3$, then the proper subgraph $S^2_3(i)\cup{S_4(i)}\cup{S_5}\cup \{v_{i-2},v_{i+2}\}$ 
    has chromatic number at least 5, contradicting that $G$ is 5-vertex-critical.
\end{proof}

\begin{claim}\label{clm:S5}
    $S_5$ is an independent set.    
\end{claim}

\begin{proof}
    If there are two adjacent vertices $u,v\in{S_5}$, then $G$ contains a $W\in \mathcal{F}$, a contradiction. 
\end{proof}

\begin{claim}\label{clm:k1k2}
    Every homogeneous component of $S^2_3(i)$ or ${S_4(i)}$ is isomorphic to $K_1$ or $K_2$.
\end{claim}

\begin{proof}
 Let $K$ be a component of $S^2_3(i)$ or ${S_4(i)}$. Since $G$ has no $K_5$ or $W$, $K$ has no triangles or $C_5$.
 Since $G$ is $P_5$-free, $G$ is bipartite. So $\chi(K)\le 2$. Clearly, if $\chi(K)=1$, then $K$ is isomorphic to $K_1$.
 Now assume that $\chi(K)=2$. Let $X$ and $Y$ be the bipartition of $K$. Let $x\in X$ and $y\in Y$ with $xy\in E$.
 Suppose that $(X\cup Y) \setminus \{x,y\}\neq \emptyset$.
 Since $G$ is 5-vertex-critical, $G-((X\cup Y) \setminus \{x,y\})$ has a 4-coloring $\phi$. Without loss of generality, we may assume that
 $\phi(x)=1$ and $\phi(y)=2$. Now if we color every vertex in $X$ with color 1 and  color every vertex in $Y$ with color 2, 
 the resulting coloring is a 4-coloring of $G$ by \autoref{clm:homo}. This contradicts that $G$ is 5-vertex-critical.
 So $K$ is isomorphic to $K_2$.
\end{proof}

\begin{claim}\label{clm:S23}
    $|S^2_3(i)|\leq3$, for all $1\leq i\leq5$.
\end{claim}

\begin{proof}
Let $K$ be a component of $S^2_3(i)$. We say that $K$ is of {\em type $i$} if $\chi(K)=i$.
We show that there is at most one component of type $i$ for $i=1,2$.
Take two components $K,K'$ of the same type. Let $k\in K$ and $k'\in K'$.
By \autoref{lem:dominated subsets}, there are vertices $u,v$ such that $u\in N(K)\setminus N(K')$ and $v\in N(K')\setminus N(K)$.
By  \autoref{clm:homo}, $uk\in E,vk'\in E$ and $uk',vk\notin E$.
Any vertex in $V(G)-(S^2_3(i)\cup{S_4(i)}\cup{S_5})$ can't mix on two vertices of $S^2_3(i)$ by \autoref{clm:other-3}. 
So $u,v\in{S_4(i)\cup S_5}$ by our assumption about $k,k'$.
If $u\nsim{v}$, $\{k,u,v_{i+1},v,k'\}$ induces a $P_5$. 
Therefore, $u\sim{v}$. By \autoref{clm:S5}, $u,v$ cannot be in $S_5$ at the same time. 
It is easy to see that $C\cup \{k,k',u,v\}$ contains an induced $P$, a contradiction.

As a result, $|S^2_3(i)|\leq 3$. 
\end{proof}

\begin{claim}\label{clm:star}
    $S_4(i)$ is a star, or $S_4(i)$ is complete to $S_4(i+2)\cup{S_4(i-2)}$, for all $1\leq{i}\leq5$.
\end{claim}
 
\begin{proof}
    If $S_4(i)$ is disconnected, $S_4(i)$ is complete to $S_4(i+2)\cup{S_4(i-2)}$ by \autoref{clm:S4(i+2)}. 
    If $S_4(i)$ is connected, then $S_4(i)$ is a bipartite graph by \autoref{clm:k1k2}.
    If $\chi(S_4(i))=1$, $S_4(i)$ is isomorphic to $K_1$ and we are done.
    Now assume that $|S_4(i)|\ge 2$. Let $X,Y$ be the bipartition of $S_4(i)$. 
    If $|X|\ge 2$ and $|Y|\ge 2$, then every vertex in $S_4(i\pm 2)$ is complete to $X\cup Y$ by \autoref{clm:S4(i+2)}.
    Thus, $S_4(i)$ is complete to $S_4(i\pm 2)$. Therefore, we may assume that $|X|=1$ and so $S_4(i)$ is a star.
\end{proof}

Recall that $R_i=S^1_3(i\pm{1})\cup S^2_3(i)\cup S_4(i\pm{1})\cup S_5$.

\begin{claim}\label{clm:size of star}
   If $S_4(i)$ is a star, then $|S_4(i)|\leq2$ for all $1\leq{i}\leq5$.
\end{claim}

\begin{proof}
Suppose that $S_4(i)=X\cup Y$ with $Y=\{y\}$. We show that $|X|\le 1$.
Suppose not. Let $x_1, x_2\in{X}$. 
By \autoref{lem:dominated subsets}, there exist $a\in{N(x_1)\backslash{N(x_2)}}$ and $b\in{N(x_2)\backslash{N(x_1)}}$. 
Note that any vertex of $G-R_i$ can't mix on two nonadjacent vertices of $X$ by \autoref{clm:S4-1} - \autoref{clm:S4(i+2)}.
So $a,b\in{R_i}$. 
If $a\nsim{b}$, $\{x_1,a,v_i,b,x_2\}$ induces a $P_5$. So $a\sim{b}$.
It is not hard to check that $G$ contains one of $Q_1$, $Q_2$ and $Q_3$, a contradiction. 
Thus, there are at most two vertices in $X$, and so $|S_4(i)|\leq2$.
\end{proof}

\begin{claim}\label{clm:S4 with Ri}
For each $i$, when $S_4(i)$ is complete to $S_4(i\pm{2})$ and $R_i$ is not empty, then $|S_4(i)|\leq6$.
\end{claim}

\begin{proof}
When $S_4(i)$ is $(P_1+P_2)$-free, $S_4(i)$ is a complete bipartite graph.
Let $(X,Y)$ be a partition of $S_4(i)$.
We show that $|X|,|Y|\leq 3$.
Suppose not. Let $x_1,x_2,x_3,x_4$ be vertices in $X$.
By \autoref{lem:dominated subsets}, there vertices $a_1\in{N(x_1)\backslash{N(x_2)}}$, $a_2\in{N(x_2)\backslash{N(x_1)}}$.
Notice that $a_1,a_2\in R_i$ by \autoref{clm:S4-1} - \autoref{clm:S4(i+2)}. 
If $a_1\nsim a_2$, $G$ contains an induced $P_5=\{x_1,a_1,v_i,a_2,x_2\}$.
So $a_1\sim a_2$.
Then $a_1\in S^1_3(i-1)\cup{S_4(i+1)}$ and $a_2\in S^1_3(i+1)\cup{S_4(i-1)}$, otherwise, it is easy to check that $G$ contains one of $Q_1$ and $Q_2$. 
Similarly, there exists $a_3\in{N(x_3)\backslash{N(x_4)}}$, $a_4\in{N(x_4)\backslash{N(x_3)}}$ and $a_3,a_4\in R_i, a_3\sim a_4$.
Thus $\{x_3,x_4\}$ is complete to $\{a_1,a_2\}$, and $\{x_1,x_2\}$ is complete to $\{a_3,a_4\}$.
This shows that $a_1,a_2,a_3,a_4$ are pairwise different vertices.
Then $a_3\in S^1_3(i-1)\cup{S_4(i+1)}$, $a_4\in S^1_3(i+1)\cup{S_4(i-1)}$.
Recall that $S^1_3(i-1)$ or $S^1_3(i+1)$ is a clique by \autoref{clm:S13}, and $S^1_3(i-1)$ is complete to $S_4(i+1)$, $S^1_3(i+1)$ is complete to $S_4(i-1)$ by \autoref{clm:S4-1}.
If $a_1\nsim a_3$ and $a_2\nsim a_4$, then $a_1,a_3\in S_4(i+1)$ and $a_2,a_4\in S_4(i-1)$, then $\{v_{i-2},v_{i+2},x_3,a_1,a_2\}$ is an induced $K_5$.
Otherwise, if $a_1\sim a_3$, $\{v_{i-1},v_{i-2},x_3,a_1,a_3\}$ induces $K_5$. 
So $a_2\sim a_4$, then $\{v_{i+1},v_{i+2},x_3,a_2,a_4\}$ induces a $K_5$, a contradiction.
So $|S_4(i)|\leq 6$ if $S_4(i)$ is $(P_1+P_2)$-free.

Now suppose that $S_4(i)$ contains a $P_1+P_2$. 
Let $P_1+P_2=\{a,b,c:a\nsim{b},a\nsim{c}, b\sim{c}\}$.
We first prove some useful facts about $P_1+P_2$.\\

\hspace{2.5cm}{${S^1_3(i)}$ is anticomplete to $P_1+P_2$.} \hfill (1)\\

Every $x\in{S^1_3(i)}$ is either complete or anticomplete to $\{a,b,c\}$ by \autoref{clm:S4-2}. If $x$ is complete to $\{a,b,c\}$, then $G$ contains an induced $W$, a contradiction. So $x$ is anticomplete to $\{a,b,c\}$. This completes the proof of (1).\\

\hspace{1.5cm}{For any $y\in{R_i}$, $\{y,a,b,c\}$ induces either a $P_4$ or a $2P_2$.}\hfill (2)\\

Let $y\in{R_i}$. 
Note that $\{y\}\cup S_4(i)$ is triangle-free or else $G$ contains a $K_5$. 
If $y$ is not adjacent to $a$, then $y\sim{b}, y\sim{c}$ by \autoref{clm:S4-3}. 
Now $G$ induces a $K_5$, a contradiction. 
So $y\sim{a}$. 
If $y\nsim{b}, y\nsim{c}$, then $\{y,a,b,c\}$ induces a $2P_2$.
If $y$ is adjacent to exact one vertex of $\{b,c\}$, we assume by symmetry that $y\sim{b}, y\nsim{c}$ and so $\{a,y,b,c\}$ induces a $P_4$. This completes the proof of (2).\\

 Next we discuss about $S_4(i)\backslash\{a,b,c\}$. Let $x\in{S^1_3(i)}$, $z\in S_4(i)\backslash\{a,b,c\}$,
 and we define $Y_1=\{y_1\in R_i: \{y_1,a,b,c\}\; \text{induces\; a\;} P_4\}$, and 
$Y_2=\{y_2\in R_i:\{y_2,a,b,c\}\; \text{induces\; a \;} 2P_2\}$.\\

\hspace{2.5cm} {$S^1_3(i)$ is anticomplete to ${S_4(i)\backslash\{a,b,c\}}$.}\hfill (3)\\

If $z\sim{x}$, then $z$ is complete to $\{a,b,c\}$ by (1). Now $G$ contains an induced $W$, a contradiction. So $z\nsim{x}$. This completes the proof of (3). 

 So ${S^1_3(i)}$ is anticomplete to $S_4(i)$ by (1) and (3).\\

\hspace{0.5cm}{For any $y_1\in Y_1, z_1\in S_4(i)\backslash\{a,b,c\}$,  $z_1y_1,z_1c\in E$, and $z_1a,z_1b\notin E$.} \hfill (4)\\

If $z_1\nsim{y_1}$, then $z_1\sim{c}$ by $y_1c\notin E$ and \autoref{clm:S4-3}. So $z_1\nsim{b}$ by \autoref{clm:clique}. If $z_1\nsim{a}$, $\{y_1,a,b,c,z\}$ induces a $P_5$. So $z_1\sim{a}$. Then there is an induced $C_5=\{a,y_1,b,c,z_1\}$, contradicting \autoref{clm:clique}.
So $z_1\sim{y_1}$, then $z_1\nsim{a}$ and $z_1\nsim{b}$ since $S_4(i)$ is triangle-free. If $z_1\nsim{c}$, $\{a,y_1,b,c\}$ and $\{z_1\}$ induce a chair. So $z_1\sim{c}$. This completes the proof (4).\\

\hspace{0.5cm}{For any $y_2\in Y_2, z_2\in S_4(i)\backslash\{a,b,c\}$, $z_2y_2\in E$, and $z_2a,z_2b,z_2c\notin E$.} \hfill (5)\\

If $z_2\nsim{y_2}$, then $z_2\sim{b}$ and $z_2\sim{c}$ by $y_2b,y_2c\notin E$ and \autoref{clm:S4-3}. 
Then $\{z_2,b,c\}$ induces a triangle, contradicting \autoref{clm:clique}.
So $z_2\sim{y_2}$ and then $z_2\nsim{a}$ by the fact that $\{y_2\}\cup S_4(i)$ is triangle-free. 
If $z_2$ is adjacent to exact one of $b,c$, then $\{z_2,y_2,a,b,c\}$ induces a $P_5$. So $z_2\nsim{b}$ and $z_2\nsim{c}$. This completes the proof (5).\\

We can infer that any vertex in $R_i$ is complete to $S_4(i)\backslash\{a,b,c\}$ by (4) and (5).
Suppose that there exist two vertices $z,z'\in S_4(i)\backslash\{a,b,c\}$.
If $Y_1\neq \emptyset$ and $Y_2\neq \emptyset$, $z$ is adjacent to $c$ by (4) and is nonadjacent to $c$ by (5), a contradiction. 
So $R_i=Y_1$ or $R_i=Y_2$.
Note that any vertex in $R_i$ is complete to two ends of an edge of $C_5\cap N(S_4(i))$.
Since $G$ is $K_5$-free, $z\nsim z'$.
Then $N(z)=N(z')$ by \autoref{clm:S4-1}, contradicting to \autoref{lem:dominated subsets}.
So $|S_4(i)\backslash\{a,b,c\}|\leq1$.
Then $|S_4(i)|\leq 4$.
\end{proof}

\begin{claim}\label{clm:empty Ri}
For each $i$, when $S_4(i)$ is complete to $S_4(i\pm{2})$ and $R_i$ is empty, $|S_4(i)|\leq2$. 
\end{claim}

\begin{proof}
If $S_4(i)$ is disconnected, then there are two components $K_1,K_2$ of $S_4(i)$.
Every vertex of $S^1_3(i)$ is either complete or anticomplete to $K_1\cup K_2$ by \autoref{clm:S4-2}. 
So $K_1$ and $K_2$ are homogeneous components by \autoref{clm:S4-1} - \autoref{clm:S4(i+2)}.
Moreover, $N(K_1)=N(K_2)\subseteq{T_i\cup{S^1_3(i)}\cup{S_4(i\pm{2})}\cup{C_5}}$. 
This contradicts \autoref{lem:dominated subsets}. Therefore, $S_4(i)$ is connected. 

Recall that $\chi(S_4(i))\leq2$ by \autoref{clm:clique}. If $\chi(S_4(i))=1$, then $|S_4(i)|=|K_1|=1$ and we are done.
When $\chi(S_4(i))=2$, $S_4(i)$ is a bipartite graph.
Let $(X,Y)$ be the bipartition of $S_4(i)$.
Every vertex $s\in S^1_3(i)$ is either complete or anticomplete to $X$(resp. $Y$) by \autoref{clm:S4-2}.
So $X$(resp. $Y$) is homogeneous with respect to $G-Y$(resp. $G-X$).
If there are $x\in X, y\in Y$ with $x\nsim y$, then every vertex $s\in S^1_3(i)$ cannot mix on $S_4(i)$.
Then $S_4(i)$ is a homogeneous set, and $|S_4(i)|=|K_2|=2$ by \autoref{clm:k1k2}.
If $X$ is complete to $Y$.
Then $X$ is a homogeneous set.
For any pairwise vertices $x_1,x_2\in X$, we have $N(x_1)=N(x_2)$, contradicting \autoref{lem:dominated subsets}.
So $|X|=1$. 
In the same way, $|Y|=1$. 
Therefore, $|S_4(i)|\leq2$.
\end{proof}

\begin{claim}\label{clm:bound S4}
    $|S_4(i)|\leq6$.
\end{claim}

\begin{proof}
It follows from \autoref{clm:size of star} to \autoref{clm:empty Ri} that $|S_4(i)|\leq6$.    
\end{proof}

\begin{claim}\label{clm:bounded}
    $|S_5|\leq2^{55}$.
\end{claim} 

\begin{proof}
     Suppose that $|S_5|>2^{55}$. We know any two vertices in $S_5$ are nonadjacent by \autoref{clm:S5}. 
     By the pigeonhole principle, there are two vertices $u,v\in{S_5}$ such that $N(u)=N(v)$, contradicting \autoref{lem:dominated subsets}. 
     So $|S_5|\leq2^{5(|S^1_3(i)\cup{S^2_3(i)}\cup{S_4(i)}|)}\\
     \leq2^{5(2+3+6)}=2^{55}$.
\end{proof}

The lemma follows from \autoref{clm:S13<3}, \autoref{clm:S23}, \autoref{clm:bound S4} and \autoref{clm:bounded}. 
\end{proof}

\section{Appendix}

Below we give the adjacency lists of graphs in $\mathcal{F}$ other than $K_5$. 

\begin{itemize}
\item Graph $W$: \{0: 1 4 5 6; 1: 0 2 5 6; 2: 1 3 5 6; 3: 2 4 5 6; 4: 0 3 5 6; 5: 0 1 2 3 4 6; 6: 0 1 2 3 4 5\}
\item Graph $P$: \{0: 1 4 5 6; 1: 0 2 7 8; 2: 1 3 5 6 7 8; 3: 2 4 5 6 7 8; 4: 0 3 7 8; 5: 0 2 3 7; 6: 0 2 3 8; 7: 1 2 3 4 5 8; 8: 1 2 3 4 6 7\}
\item Graph $Q_1$: \{0: 1 4 5 6; 1: 0 2 5 6 7 8; 2: 1 3 5 6 7 8; 3: 2 4 7 8; 4: 0 3 7 8; 5: 0 1 2 6 7; 6: 0 1 2 5 8; 7: 1 2 3 4 5; 8: 1 2 3 4 6\}
\item Graph $Q_2$: \{0: 1 4 5 6; 1: 0 2 5 6 7 8; 2: 1 3 5 6 7 8; 3: 2 4 5 6 7 8; 4: 0 3 7 8; 5: 0 2 3 6 7; 6: 0 2 3 5 8; 7: 1 2 3 4 5; 8: 1 2 3 4 6\}
\item Graph $Q_3$: \{0: 1 4 5 6; 1: 0 2 5 7 8; 2: 1 3 5 7 8; 3: 2 4 6 7 8; 4: 0 3 6 7 8; 5: 0 1 2 6; 6: 0 3 4 5 8; 7: 1 2 3 4 8; 8: 1 2 3 4 6 7\}
\end{itemize}


\end{document}